\documentclass[preprint]{elsarticle}
\usepackage[letterpaper, left=2cm, right=2cm, top=2cm, bottom=2cm]{geometry}
\usepackage{hyperref}
\usepackage{url}
\usepackage{breakurl}
\usepackage{graphicx,color}
\usepackage{paralist}
\usepackage{amsmath, amsthm, amssymb}
\usepackage[square]{natbib}

\newtheorem{thm}{Theorem}
\newtheorem{lem}{Lemma}

\newcommand{\td}{\tilde D}
\newcommand{\bx}{{\bf x}}
\newcommand{\by}{{\bf y}}
\newcommand{\rn}{\mathbb{R}^n}
\newcommand{\rr}{\mathbb{R}}
\newcommand{\sn}{S_{2}^{n-1}}
\newcommand{\lp} {\lambda'}

\begin{document}

\title{On Euclidean Norm Approximations}
\author{M.\ Emre Celebi}
\address{Department of Computer Science\\Louisiana State University, Shreveport, LA, USA\\
        \href{mailto:ecelebi@lsus.edu}{ecelebi@lsus.edu}}
\author{Fatih Celiker}
\address{Department of Mathematics\\Wayne State University, Detroit, MI, USA\\
         \href{mailto:celiker@math.wayne.edu}{celiker@math.wayne.edu}}
\author{Hassan A.\ Kingravi}
\address{Department of Computer Science\\Georgia Institute of Technology, Atlanta, GA, USA\\
         \href{mailto:kingravi@gatech.edu}{kingravi@gatech.edu}}

\begin{abstract}
Euclidean norm calculations arise frequently in scientific and engineering applications.
Several approximations for this norm with differing complexity and accuracy have been proposed in the literature.
Earlier approaches \cite{Chaudhuri92,Rhodes95,Barni95} were based on minimizing the maximum error. Recently,
Seol and Cheun \cite{Seol08} proposed an approximation based on minimizing the average error. In this paper,
we first examine these approximations in detail, show that they fit into a single
mathematical formulation, and compare their average and maximum errors.
We then show that the maximum errors given
by Seol and Cheun are significantly optimistic.

\end{abstract}

\maketitle

\section{Introduction}
\label{sec_intro}
The Minkowski ($L_p$) metric is inarguably one of the most commonly used
quantitative distance (dissimilarity) measures in
scientific and engineering applications.
The Minkowski distance between two vectors
$\bx = (x_1 ,x_2 , \ldots ,x_n )$ and {$\by = (y_1 ,y_2 , \ldots ,y_n )$}
in the $n$-dimensional Euclidean space, $\mathbb{R}^n$, is given by
\begin{equation}
\label{equ_lp}
 L_p (\bx,\by) =
 \left( {\sum\nolimits_{i = 1}^n {\left| {x_i  - y_i } \right| ^p } } \right)^{1/p}.
\end{equation}
Three special cases of the $L_p$ metric are of particular interest, namely,
$L_1$ (city-block metric), $L_2$ (Euclidean metric), and $L_\infty$ (chessboard metric).
Given the general form (\ref{equ_lp}), $L_1$ and $L_2$ can be defined
in a straightforward fashion, while $L_\infty$ is defined as
\begin{equation*}
L_{\infty}({\bf x},{\bf y}) = \max_{1\le i\le n}|x_i-y_i|.
\end{equation*}
In many applications, the data space is Euclidean and therefore the $L_2$ metric is the natural choice.
In addition, this metric has the advantage of being isotropic (rotation invariant). For example, when
the input vectors stem from an isotropic vector field, e.g.\ a velocity field, the most appropriate choice is
to use the $L_2$ metric so that all vectors are processed in the same way, regardless of their orientation
\cite{Barni95}.

The main drawback of $L_2$ is its high computational requirements due to
the multiplications and the square root operation.
As a result, $L_1$ and $L_\infty$ are often used as alternatives.
Although these metrics are computationally more efficient,
they deviate from $L_2$ significantly.
The Minkowski metric is translation invariant, i.e.\ $L_p(\bx,\by) = L_p(\bx + \bf{z},\by + \bf{z})$ for all $\bx, \by, \bf{z} \in \mathbb{R}^n$, hence it suffices to consider $D_p(\bx) = L_p (\bx,{\bf 0})$,
i.e. the distance from the point $\bx$ to the origin.
Therefore, in the rest of the paper, we will consider approximations to
$D_p(\bx)$ rather than $L_p (\bx,\by)$.

In this paper, we examine several approximations to the Euclidean norm.
The rest of the paper is organized as follows. In Section \ref{sec_approx}
we describe the Euclidean norm approximations that have appeared in the literature,
and compare their average and maximum errors using numerical simulations.
We then show that all of these methods fit into a single mathematical formulation.
In Section \ref{sec_sampling} we examine the simulation results from a theoretical
perspective. Finally, in Section \ref{sec_conc} we provide our conclusions.

\section{Euclidean Norm Approximations}
\label{sec_approx}
For reasons explained in Sec. \ref{sec_intro}, we concentrate on approximations
to the Euclidean norm $D_2$ on $\mathbb{R}^n$.
Let $\td$, defined on $\rn$, be an approximation to $D_2$.
We assume that $\td$ is a continuous homogeneous function.
We note that all variants of $\td$ we consider in this paper
satisfy these assumptions.
As a measure of the quality of the approximation of $\td$ to $D_2$
we define the maximum relative error (MRE) as
\begin{equation}
\label{equ_max_err_temp1}
\varepsilon_{\text{max}}^{\td} = \sup \limits_{\bx \in \rn\setminus\{{\bf 0}\}}
{\frac{|{\td(\bx)-D_2(\bx)}|}{{D_2(\bx)}}}.
\end{equation}
Using the homogeneity of $D_2$ and $\td$,
\eqref{equ_max_err_temp1} can be written as
\begin{equation}
\label{equ_max_err_temp2}
\varepsilon_{\text{max}}^{\td} = \sup_{\bx\in\sn} |\td(\bx)-1|,
\end{equation}
where
\[
\sn = \{ \bx \in \rn : D_2(\bx)=1 \}
\]
is the unit hypersphere of $\rn$ with respect to the Euclidean norm.
Furthermore, by the continuity of $\td$, we can replace the supremum with maximum
in \eqref{equ_max_err_temp2} and write
\begin{equation}
\label{equ_max_err}
\varepsilon_{\text{max}}^{\td} = \max_{\bx\in\sn} |\td(\bx)-1|.
\end{equation}
We will use \eqref{equ_max_err} as the definition of MRE throughout.

In the trivial case where $\td=D_2$ we have $\varepsilon_{\text{max}}^{\td}=0$.
Hence, for nontrivial cases we wish to have a small $\varepsilon_{\text{max}}^{\td}$ value.
In other words, the smaller the value of $\varepsilon_{\text{max}}^{\td}$,
the better (more accurate) the corresponding approximation ${\td}$.
It can be shown that $D_1$ (city-block norm)
overestimates $D_2$ and the corresponding MRE is given by
$\varepsilon_{\text{max}}^{D_1} = \sqrt{n}-1$ \cite{Chaudhuri92}.
In contrast, $D_\infty$ (chessboard norm) underestimates $D_2$ with MRE given by
$\varepsilon_{\text{max}}^{D_{\infty}} = 1-1/\sqrt{n}$ \cite{Chaudhuri92}.
More explicitly,
\begin{equation}\label{equ_bounds}
\begin{aligned}
D_2(\bx)             &\leq D_1(\bx)\;    \leq \sqrt{n}D_2(\bx)\\
(1/\sqrt{n})D_2(\bx) &\leq D_\infty(\bx) \leq D_2(\bx)
\end{aligned}
\end{equation}
for all $\bx \in \rn$.
Therefore, it is natural to expect a suitable linear combination of $D_1$ and $D_\infty$
to give an approximation to $D_2$ better than both $D_1$ and $D_\infty$ \cite{Rhodes95}.

\subsection{Chaudhuri \emph{et al.}'s approximation}
\label{sec_chaudhuri}
Chaudhuri \emph{et al.} \cite{Chaudhuri92} proposed the approximation
\footnote{Unfortunately, the motivation behind this particular choice of $\lambda$ is not given in the paper.}
\begin{equation*}
D_{\lambda}(\bx) =
|x_{i_{\text{max}}}| +
\lambda\hspace{-.1in}
       \sum_{\begin{subarray}{l}
             \quad i =1 \\
             \quad i \ne i_{\max}
             \end{subarray}}^n \hspace{-.1in}|x_i|,
\qquad\mbox{with}\qquad
\lambda  = \frac{1}{{n - \left\lfloor{\frac{n-2}{2}}\right\rfloor}}.
\end{equation*}
Here $i_{\text{max}}$ is the index of the absolute largest component of $\bx$,
i.e.\ $i_{\text{max}} = \mathop {\arg \max }\limits_{1 \leq i \leq n} \left( |{x_i}| \right)$,
and $\lfloor x \rfloor$ is the floor function which returns
the largest integer less than or equal to $x$.
Since $D_\infty(\bx) = |x_{i_{\text{max}}}|$, by adding and subtracting
the term $\lambda|x_{i_{\text{max}}}|$, $D_{\lambda}$ can be written
as a linear combination of $D_\infty$ and $D_1$ as
\begin{equation}
\label{equ_lin_comb}
D_\lambda(\bx) = (1 - \lambda)D_\infty(\bx) + \lambda D_1(\bx).
\end{equation}
It is easy to see that $D_\infty(\bx) \leq D_\lambda(\bx) \leq D_1(\bx)$
 for all $\bx \in \mathbb{R}^n$ since $0 < \lambda \leq 0.5$.
It can also be shown that \cite{Chaudhuri92} for sufficiently large $n$, $D_\lambda$ is closer to $D_2$ than both $D_1$ and $D_\infty$, i.e.\ $\left| D_\lambda(\bx) - D_2(\bx) \right| \leq \left| D_1(\bx) - D_2(\bx) \right|$ and
$\left| D_\lambda(\bx) - D_2(\bx) \right| \leq \left| D_2(\bx) - D_\infty(\bx) \right|$ for all $\bx \in \mathbb{R}^n$.

For sufficiently large $n$, $D_{\lambda}$ underestimates $D_2$
and the corresponding MRE is given by
\begin{equation*}
1 - \frac{{1-\lambda(n-1)}}{{\sqrt n}}
\quad\le\quad
\varepsilon_{\text{max}}^{D_{\lambda}}
\quad\le\quad
1 - \lambda
\quad
\mbox{for} \quad n \geq 3.
\end{equation*}
Otherwise, $D_{\lambda}$ overestimates $D_2$ and we have
\begin{equation*}
\varepsilon_{\text{max}}^{D_{\lambda}} =
\left\{
 \begin{array}{l}
 \sqrt {1 + \frac{{4(n - 1)}}{{(n + 2)^2 }}}  - 1\quad {\rm for}\;n = 2,4,6, \ldots  \\
 \\
 \sqrt {1 + \frac{{4(n - 1)}}{{(n + 3)^2 }}}  - 1\quad {\rm for}\;n = 3,5,7, \ldots  \\
 \end{array}
\right.
\end{equation*}
Proofs of these identities can be found in \cite{Chaudhuri92}.

\subsection{Rhodes' approximations}
\label{sec_rhodes}

Rhodes \cite{Rhodes95} reformulated (\ref{equ_lin_comb}) as a maximum of linear functions
\begin{equation*}
D_\lambda(\bx) = \mathop {\max }\limits_{1 \leq j \leq n} \left\{ {(1 - \lambda )|x_j| +
\lambda \sum\limits_{i = 1}^n {|x_i|} } \right\}
\end{equation*}
where $0 < \lambda < 1$ is taken as a free parameter. He determined the optimal value for $\lambda$
by minimizing $\varepsilon_{\max}^{D_{\lambda}} = \max_{\bx \in S_2^{n-1}}|D_{\lambda}(\bx)-1|$ analytically.
In particular, he showed that optimal $\lp$ values for
Chaudhuri \emph{et al.}'s norm can be determined by
solving the equation $1 - 2\sqrt{\lp - (\lp) ^2} = \sqrt {1 + (\lp)^2 (n - 1)} - 1$
in the interval $(0,1/2)$.
This equation is a quartic (fourth order) in $\lambda$
and can be solved using Ferrari's method \cite{King08}.
It can be shown that this particular quartic equation has two real and two complex roots and
the optimal $\lambda'$ value is given by the smaller of the real roots.
The corresponding MRE is given by \cite{Rhodes95}
\begin{equation}
\label{equ_chaudhuri_mre}
\varepsilon_{\max}^{D_{\lambda}} = 1-2\sqrt{\lp-(\lp)^2}.
\end{equation}
In the remainder of this paper, $D_\lambda$ refers to this improved
variant of Chaudhuri \emph{et al.}'s norm.

Rhodes also investigated the two-parameter family of approximations given by
\begin{equation}
\label{equ_rhodes_two}
D_{\mu ,\lambda} (\bx) = (\mu - \lambda )D_\infty(\bx) + \lambda D_1(\bx)
\end{equation}
where $0 < \lambda < \mu$.
He proved that the optimal solution and its MRE in this case are given by
\begin{equation}
\label{equ_rhodes_mre}
\begin{aligned}
\lambda^* &= \frac{2}{2n^{1/4}+\sqrt{2n+2\sqrt n}}, \qquad
\mu^*     = \left( {\sqrt{n}+1}\right)\lambda^*, \\
&\qquad\qquad\qquad
\varepsilon_{\max}^{D_{\mu,\lambda}} = 1-2\lambda^*n^{1/4}.
\end{aligned}
\end{equation}

Finally, Rhodes investigated the $D_{\mu ,\lambda}$
approximations with $0 \leq \mu <\lambda$.
He proved that the optimal solution and its MRE are given by
\begin{equation*}
\lambda^* = \frac{2}{1+\sqrt{n-1}}, \qquad
\mu^*=0, \qquad
\varepsilon_{\max}^{D_{\mu,\lambda}} = 1-\lambda^*.
\end{equation*}
This approximation will not be considered any further since its accuracy
is inferior to even the single-parameter approximation $D_\lambda$.

It should be noted that Rhodes optimized $D_\lambda$ and $D_{\mu ,\lambda}$ over $\mathbb{Z}^n$.
Therefore, these norms are in fact suboptimal on $\mathbb{R}^n$ (see \S \ref{sec_comparisons}).

\subsection{Barni \emph{et al.}'s approximation}
\label{sec_barni}
Barni \emph{et al.} \cite{Barni95,Barni00} formulated a generic approximation for $D_2$ as
\begin{equation*}
D_B(\bx) = \delta \sum\limits_{i = 1}^n {\alpha _i x_{(i)}}
\end{equation*}
where $x_{(i)}$ is the $i$-th absolute largest component of
$\bx$, i.e. $(x_{(1)}, x_{(2)}, \cdots, x_{(n)})$ is a permutation of
$(|x_1|,|x_2|,\cdots,|x_n|)$ such that
$x_{(1)} \geq x_{(2)} \geq \ldots \geq x_{(n)}$.
Here $\boldsymbol{\alpha} = (\alpha_1, \alpha_2, \cdots, \alpha_n)$
and $\delta > 0$ are approximation parameters.
Note that a non-increasing ordering and strict positivity of the component weights,
i.e. $\alpha_1 \ge \alpha_2 \ge \cdots \ge \alpha_n > 0$
is a necessary and sufficient condition for $D_B$ to define a norm \cite{Barni00}.

The minimization of (\ref{equ_max_err}) is equivalent to determining the weight vector $\boldsymbol{\alpha}$
and the scale factor $\delta$ that solve the following minimax problem \cite{Barni00,Demyanov90}
\begin{equation}
\min\limits_{\boldsymbol{\alpha},\delta} \max\limits_{\bx \in V} \left| {D_B(\bx) - 1} \right|
\end{equation}
where $V = \{ \bx \in \rn \,:\, x_1 \ge x_2 \ge \cdots \ge x_n \ge 0,\; D_2(\bx)=1 \}$.

The optimal solution and its MRE are given by
\begin{equation}
\label{equ_barni_mre}
\alpha_i^* = \sqrt{i}-\sqrt{i-1}, \qquad
\delta^* = \frac{2}{{1+\sqrt{\sum_{i=1}^n {{\alpha_i^*}^2}}}}, \qquad
\varepsilon_{\text{max}}^{D_B}= 1 - \delta^*.
\end{equation}
It should be noted that a similar but less rigorous approach had been published earlier by Ohashi \cite{Ohashi94}.

\subsection{Seol and Cheun's approximation}
\label{sec_seol}

Seol and Cheun \cite{Seol08} recently proposed an approximation of the form
\begin{equation}
\label{equ_seol_cheun}
D_{a,b} (\bx) = a D_\infty(\bx) + b D_1(\bx)
\end{equation}
where $a$ and $b$ are strictly positive parameters
to be determined by solving the following $2\times 2$ linear system
\begin{alignat*}{3}
 &a{\rm E}({D_\infty^2}  ) &&+ b{\rm E}({D_\infty D_1}) &&= {\rm E}({D_2 D_\infty}), \\
 &a{\rm E}({D_\infty D_1}) &&+ b{\rm E}({D_1^2}       ) &&= {\rm E}({D_2 D_1}),
\end{alignat*}
where $\rm E (\cdot)$ is the expectation operator.

Note that the formulation of $D_{a,b}$ is similar
to that of $D_{\mu ,\lambda}$ \eqref{equ_rhodes_two}
in that they both approximate $D_2$ by a linear combination of
$D_{\infty}$ and $D_1$. These approximations differ in their
methodologies for finding the optimal parameters. Rhodes follows
an analytical approach and derives theoretical values for the parameters
and the maximum error. However, he achieves this by sacrificing
maximization over $\rn$, and maximizes only over $\mathbb{Z}^n$.
Seol and Cheun follow an empirical approach where they approximate
optimal parameters over $\rn$, which causes them to sacrifice the ability
to obtain analytical values for the parameters and the maximum error.
They estimate the optimal values of $a$ and $b$ using
$100,000$ $n$-dimensional vectors whose components are independent and
identically distributed, standard Gaussian random variables.

\subsection{Comparison of the Euclidean norm approximations}
\label{sec_comparisons}
It is easy to see that all of the presented approximations fit into the general form
\begin{equation*}
\td(\bx) = \sum\limits_{i=1}^n {w_i x_{(i)}}
\end{equation*}
which is a weighted $D_1$ norm.

The component weights for each approximation are given in Table \ref{tab_weights}.
It can be seen that $D_B$ has the most elaborate design in which each component
is assigned a weight proportional to its ranking. However, this weighting scheme
also presents a drawback in that a full ordering of the component absolute values
is required (see Table \ref{tab_cost}).

\begin{table}[h]
\centering
\caption{ \label{tab_weights} Weights for the approximate norms }
\vspace{.2cm}
\begin{tabular}{c|c|c}
\hline
Norm  & $w_1$ & $w_{i \neq 1}$\\
\hline
\hline
$D_\lambda$ & 1 & $\lambda'$\\
\hline
$D_{\mu ,\lambda}$ & $\mu^*$ & $\lambda^*$\\
\hline
$D_B$ & $\delta^*$ & $\delta^* \alpha_i^*$\\
\hline
$D_{a,b}$ & $a + b$ & $b$\\
\hline
\end{tabular}
\end{table}

Due to their formulations, the MRE's for
$D_\lambda$, $D_{\mu ,\lambda}$, and $D_B$
can be calculated analytically using \eqref{equ_chaudhuri_mre}, \eqref{equ_rhodes_mre},
and \eqref{equ_barni_mre}, respectively.
In Figure \ref{fig_max_err} we plot the theoretical errors for these norms for $n \leq 100$.
It can be seen that $D_B$ is not only more accurate than
$D_\lambda$ and $D_{\mu ,\lambda}$, but it also scales significantly better.
Although $D_{\mu ,\lambda}$ is more accurate than $D_\lambda$ when $n$ is small,
the difference between the two approximations becomes less significant as $n$ is increased.
\begin{figure}[!ht]
\centering
\includegraphics[width=0.72\columnwidth]{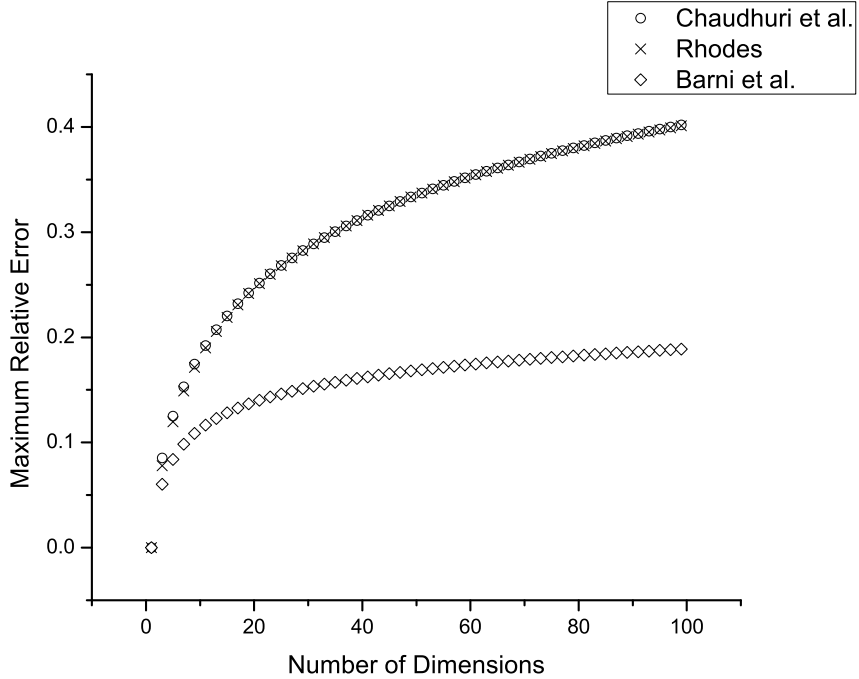}
\caption{\label{fig_max_err} Maximum relative errors for $D_\lambda$, $D_{\mu ,\lambda}$, and $D_B$}
\end{figure}

The operation counts for each norm are given in Table \ref{tab_cost}
(\textbf{ABS}: absolute value,
\textbf{COMP}: comparison,
\textbf{ADD}: addition,
\textbf{MULT}: multiplication,
\textbf{SQRT}: square root).
The following conclusions can be drawn:
\begin{itemize}
\renewcommand{\labelitemi}{$\triangleright$}
\item
$D_B$ has the highest computational cost among the approximate norms
due to its costly weighting scheme, which requires sorting of $n$
numbers and $n$ multiplications.
For small values of $n$, sorting can be performed most efficiently
by a sorting network \cite{Cormen09}.
For large values of $n$, sorting requires $\mathcal{O}(n \log{n})$ comparisons,
which is likely to exceed the cost of the square root operation \cite{Barni00}.
Therefore, in high dimensional spaces, e.g. $n > 9$ \cite{Seol08}, $D_B$ provides
no computational advantage over $D_2$.
\item
$D_\lambda$ has the lowest computational cost among the approximate norms.
$D_{\mu ,\lambda}$ and $D_{a,b}$ have the same computational cost, which is
slightly higher than that of $D_\lambda$.
\item
A significant advantage of $D_\lambda$, $D_{\mu ,\lambda}$, and $D_{a,b}$
is that they require a fixed number of multiplications (1 or 2) regardless of
the value of $n$.

\item
$D_\lambda$, $D_{\mu ,\lambda}$, and $D_{a,b}$ can be used to approximate
$D_2^2$ (squared Euclidean norm) using an extra multiplication.
On the other hand, the computational cost of $D_B$ is higher than that of
$D_2^2$ due to the extra absolute value and sorting operations
involved.
\end{itemize}

\begin{table}
\centering
\caption{ \label{tab_cost} Operation counts for the norms}
\vspace{.2cm}
\begin{tabular}{c|c|c|c|c|c}
\hline
\textbf{Norm} & \textbf{ABS} & \textbf{COMP} & \textbf{ADD} & \textbf{MULT} & \textbf{SQRT}\\
\hline
\hline
$D_\infty$ & $n$ & $n - 1$ & 0 & 0 & 0\\
\hline
$D_1$ & $n$ & 0 & $n - 1$ & 0 & 0\\
\hline
$D_2$ & 0 & 0 & $n - 1$ & $n$ & 1\\
\hline
$D_\lambda$ & $n$ & $n - 1$ & $n - 1$ & 1 & 0\\
\hline
$D_{\mu ,\lambda}$ & $n$ & $n - 1$ & $n$ & 2 & 0\\
\hline
$D_B$ & $n$ & $\mathcal{O}(n \log{n})$ & $n - 1$ & $n$ & 0\\
\hline
$D_{a,b}$ & $n$ & $n - 1$ & $n$ & 2 & 0\\
\hline
\end{tabular}
\end{table}

In Table \ref{tab_avg_max_err} we display the average and maximum errors
for $D_\lambda$, $D_{\mu ,\lambda}$, and $D_B$ for $n \leq 10$.
Average relative error (ARE) is defined as
\begin{equation}
\label{equ_avg_err}
\varepsilon_{\text{avg}}^{\td} =
\frac{1}{|S|} \sum\limits_{\bx \in S}
{|\td(\bx) - 1|}
\end{equation}
where $S$ is a finite subset of the unit hypersphere $S_2^{n-1}$,
and $|S|$ denotes the number of elements in $S$.
An efficient way to pick a random
point on $S_2^{n-1}$ is to generate $n$ independent Gaussian random variables
$x_1, x_2, \ldots, x_n$ with zero mean and unit variance.
The distribution of the unit vectors
\[
\Big\{\by=(y_1,y_2,\ldots,y_n)\; : \; y_i={{x_i}/{\Big(\sum\nolimits_{i=1}^{n}x_i^2\Big)^{1/2},\quad i=1,2,\ldots,n}}\Big\}
\]
will then be uniform over the surface of the hypersphere \cite{Muller59}. For each approximate norm,
the ARE and MRE values were calculated over an increasing number of points,
$ 2^{20}, 2^{21}, \ldots, 2^{32} - 1 $ (that are uniformly distributed on the
hypersphere) until the error values converge, i.e.\ the error values do not differ
by more than $\epsilon = 10^{-5}$ in two consecutive iterations. Note that for each norm,
two types of maximum error were considered: empirical maximum error ($\mbox{MRE}_e$), which
is calculated numerically over $S$ and the theoretical maximum error ($\mbox{MRE}_t$),
which is calculated analytically using \eqref{equ_chaudhuri_mre}, \eqref{equ_rhodes_mre},
or \eqref{equ_barni_mre}.
It can be seen that for $D_B$, the empirical and maximum errors agree in all cases,
which demonstrates the validity of the presented iterative error calculation scheme.
This is not the case for $D_\lambda$ and $D_{\mu ,\lambda}$ since these norms are optimized
over $\mathbb{Z}^n$ instead of $\mathbb{R}^n$.
Therefore, a perfect agreement between the empirical and theoretical results should not be expected.
Nevertheless, the empirical error is always less than the
maximum error, which is expected because we are maximizing over a smaller set.

Table \ref{tab_seol_err} shows the average and maximum errors for $D_{a,b}$.
The error values under the column ``Seol \& Cheun" are taken from \cite{Seol08}
(where the simulations were performed on a set of $100,000$ $n$-dimensional vectors
whose components are independent and identically distributed, zero mean, and unit
variance Gaussian random variables), whereas those under the column ``This study"
were obtained using the aforementioned iterative scheme. It can be seen that
the maximum errors obtained by Seol \& Cheun are lower than those that we obtained
and the discrepancy between the outcomes of the two error calculation schemes
increases as $n$ is increased. The optimistic maximum error values given by
Seol and Cheun are due to the fact that $100,000$ vectors are not enough to cover
the surface of the hypersphere in higher dimensions. This is investigated further in
the following section. On the other hand, the average error values agree perfectly in both
calculation schemes.

\begin{table}
\centering
\caption{ \label{tab_avg_max_err} Average and maximum errors for $D_\lambda$,
$D_{\mu ,\lambda}$, and $D_B$ }
\vspace{.2cm}
\begin{tabular}{c|c|c|c|c|c|c|c|c|c}
\hline
 & \multicolumn{3}{|c|}{$D_\lambda$} & \multicolumn{3}{|c|}{$D_{\mu ,\lambda}$}
 & \multicolumn{3}{|c}{$D_B$}\\
\hline
 $n$ & ARE & $\mbox{MRE}_e$ & $\mbox{MRE}_t$ & ARE & $\mbox{MRE}_e$ & $\mbox{MRE}_t$
 & ARE & $\mbox{MRE}_e$ & $\mbox{MRE}_t$\\
\hline
2 & 0.0348 & 0.0551 & 0.0551 & 0.0276 & 0.0470 & 0.0470 & 0.0241 & 0.0396 & 0.0396\\
\hline
3 & 0.0431 & 0.0852 & 0.0852 & 0.0367 & 0.0778 & 0.0778 & 0.0300 & 0.0602 & 0.0602\\
\hline
4 & 0.0455 & 0.1074 & 0.1074 & 0.0420 & 0.1010 & 0.1010 & 0.0345 & 0.0739 & 0.0739\\
\hline
5 & 0.0460 & 0.1251 & 0.1251 & 0.0447 & 0.1197 & 0.1197 & 0.0377 & 0.0839 & 0.0839\\
\hline
6 & 0.0458 & 0.1400 & 0.1400 & 0.0462 & 0.1354 & 0.1354 & 0.0401 & 0.0919 & 0.0919\\
\hline
7 & 0.0454 & 0.1529 & 0.1529 & 0.0469 & 0.1489 & 0.1490 & 0.0418 & 0.0984 & 0.0984\\
\hline
8 & 0.0448 & 0.1641 & 0.1643 & 0.0471 & 0.1606 & 0.1609 & 0.0431 & 0.1039 & 0.1039\\
\hline
9 & 0.0442 & 0.1739 & 0.1745 & 0.0471 & 0.1709 & 0.1716 & 0.0440 & 0.1086 & 0.1086\\
\hline
10 & 0.0435 & 0.1827 & 0.1837 & 0.0469 & 0.1803 & 0.1812 & 0.0447 & 0.1128 & 0.1128\\
\hline
\end{tabular}
\end{table}

\begin{table}
\centering
\caption{ \label{tab_seol_err} Average and maximum errors for $D_{a,b}$ }
\vspace{.2cm}
\begin{tabular}{c|c|c|c|c}
\hline
 & \multicolumn{2}{|c|}{Seol \& Cheun} & \multicolumn{2}{|c}{This study}\\
\hline
 $n$ & ARE & $\mbox{MRE}_e$ & ARE & $\mbox{MRE}_e$\\
\hline
2 & 0.0200 & 0.0526 & 0.0200 & 0.0525\\
\hline
3 & 0.0239 & 0.0991 & 0.0239 & 0.0998\\
\hline
4 & 0.0257 & 0.1342 & 0.0257 & 0.1363\\
\hline
5 & 0.0268 & 0.1420 & 0.0268 & 0.1649\\
\hline
6 & 0.0273 & 0.1674 & 0.0273 & 0.1871\\
\hline
7 & 0.0276 & 0.1772 & 0.0276 & 0.1968\\
\hline
8 & 0.0277 & 0.1753 & 0.0277 & 0.2076\\
\hline
9 & 0.0277 & 0.1711 & 0.0277 & 0.2120\\
\hline
10 & 0.0276 &   0.1526 & 0.0276 & 0.2156\\
\hline
\end{tabular}
\end{table}

By examining Tables \ref{tab_avg_max_err} and \ref{tab_seol_err}
the following observations can be made regarding the maximum error:

\begin{itemize}
\renewcommand{\labelitemi}{$\triangleright$}
\item
$D_B$ is the most accurate approximation in all cases.
This is because this norm is designed to minimize the
maximum error and it has a more sophisticated weighting
scheme than the other two approximations, i.e.\ $D_\lambda$
and $D_{\mu ,\lambda}$, that are based on the same optimality criterion.
\item
As is also evident from Figure \ref{fig_max_err},
$D_{\mu ,\lambda}$ is slightly more accurate than
$D_\lambda$ especially for small values of $n$, in accordance with the
greater degrees of freedom it is afforded.

\item
$D_{a,b}$ is the least accurate approximation except for $n=2$.
This was expected since this norm is designed to minimize
the mean squared error rather than the maximum error.
\item
As $n$ is increased, the error increases in all approximations.
However, as can be seen from Figure \ref{fig_max_err},
the error grows faster in some approximations than others.
\end{itemize}

On the other hand, with respect to average error we can see that:

\begin{itemize}
\renewcommand{\labelitemi}{$\triangleright$}
\item
As expected, $D_{a,b}$ is the most accurate approximation.
\item
As $n$ is increased, the error increases consistently for the $D_B$ norm.
This is not the case for the $D_\lambda$ and $D_{\mu ,\lambda}$ norms.
This inconsistent average error behavior is not surprising given the fact
that these norms are designed to minimize the maximum error.
\item
Interestingly, $D_\lambda$ is more accurate than $D_{\mu,\lambda}$ for $n > 5$.
A possible explanation to this phenomenon is that both approximations are optimized
for the maximum error. Since the minimization of the maximum and average errors are
conflicting objectives, it is likely that $D_{\mu ,\lambda}$ sacrifices the average
error to obtain better (lower) maximum error. The same relationship holds between
$D_{a,b}$ and $D_B$.
\end{itemize}

\section{Sampling on the Unit Hypersphere}
\label{sec_sampling}

In this section, we demonstrate why a fixed number of samples from the unit hypersphere (i.e. the
approach advocated in \cite{Seol08}) can give biased estimates for the maximum error. The basic reason
behind this is the fact that a fixed number of samples fail to suffice as the dimension of the space
increases. The following provides a plausibility argument as to why this is the case.
To this end, we need to consider the notion of covering a sphere `sufficiently'.
We begin with some definitions.

A closed $n$-ball of radius $r$ with respect to the Euclidean norm,
denoted $B_2^n(r)$, is the set of points whose Euclidean norm is less than or equal to $r$.
That is,
\[
B_2^n(r) = \{ \bx \in \rn : D_2(\bx) \le r \}.
\]
Note that, in particular, the unit hypersphere $S_2^{n-1}$ of $\rn$ is the {\em boundary} of $B_2^n(1)$.

Given an $\varepsilon > 0$, we say that a set $\mathcal{C}$ of points on $S_2^{n-1}$ is
an $\varepsilon$-\emph{dense covering} of $S_2^{n-1}$
if for any $\bx$ in $\mathcal{C}$, there exists at least one $\hat{\bx}$ in $\mathcal{C}$
(different than $\bx$) such that $D_2(\bx-\hat{\bx}) < \varepsilon$.
Essentially, our main purpose here is to give a rough estimate of the number of points in $\mathcal{C}$,
where $\mathcal{C}$ is an $\varepsilon$-\emph{dense covering} of $S_2^{n-1}$.
We would then argue that, if $\varepsilon$ is sufficiently small then $\mathcal{C}$ is a fine-enough
representation of points on $S_2^{n-1}$. Therefore, we can restrict any computation that
needs to be performed on $S_2^{n-1}$ to the finite set $\mathcal{C}$.

The basic idea behind the proof is to approximate $S_2^{n-1}$ by $B_2^{n-1}(\varepsilon)$,
that is, approximate the unit hypersphere of $\rn$ by $(n-1)$-balls of radius $\varepsilon$.
This is the same principle as approximating a circle $(S_2^1)$ in $\rr^2$
by tiny line segments $(B_2^1(\varepsilon))$,
or the surface of a sphere $(S_2^2)$ in $\rr^3$ by tiny discs $(B_2^2(\varepsilon))$.
It is easy to see that, if we choose $\varepsilon$ small enough, then the approximation
is satisfactory for most practical purposes.

To proceed further, we need a lemma from elementary probability theory,
which is known as the \emph{coupon collector's problem} \cite{Mitzenmacher05}.
\begin{lem}
\label{lem_coupon}
Given a collection of $c$ distinct objects, the expected number of independent random trials needed to
sample each one of the $c$ objects is $\mathcal{O}(c \log c)$.
\end{lem}

We can now prove the following result.
\begin{thm}
\label{thm_1}
The expected number of uniformly distributed samples needed to
generate an $\varepsilon$-dense covering of $S_2^{n-1}$ is
$\mathcal{O}(N \log N)$ where $N = \frac{n}{\varepsilon^{n-1}}$
\end{thm}
\begin{proof}
Let $\varepsilon > 0$ be given.
We will first count the number of identical copies of $B_2^{n-1}(\varepsilon)$-balls
that are needed to approximate $S_2^{n-1}$ in the sense described above.
By elementary calculus one can compute the volume of an $B_2^{n}(r)$ to be
\[
V_n(r) := \frac{\pi^{\frac{n}{2}}}{\Gamma\left({\frac{n}{2} + 1}\right)}r^n := C_n r^n,
\]
where $\Gamma$ is the gamma function.
The surface area of this ball is equal to the derivative with respect to $r$ of its volume:
\[
A_{n-1}(r) = \frac{d}{dr}V_n(r) = nC_nr^{n-1}.
\]
Note that $A_{n-1}(1)$ is equal to the surface area of $S_2^{n-1}$.
The approximate number of $B_2^{n-1}(\varepsilon)$-balls
needed to cover the surface of $S_2^{n-1}$ is the ratio of the surface area of $S_2^{n-1}$
to the volume of $B_2^{n-1}(\varepsilon)$, i.e.,
\begin{alignat*}{1}
\frac{A_{n-1}(1)}{V_{n-1}(\varepsilon)}
&= \frac{n\,C_n}{C_{n-1} \varepsilon^{n-1}}
= \frac{n\,\pi^{\frac{n}{2}}}{\pi^{\frac{n-1}{2}}}
\frac{\Gamma\left({\frac{n-1}{2}+1}\right)}{\Gamma\left({\frac{n}{2}+1}\right)}
\frac{1}{\varepsilon^{n-1}} \\
&\approx \frac{n\,\pi^{\frac12}}{\varepsilon^{n-1}}
= \mathcal{O}\left(\frac{n}{\varepsilon^{n-1}}\right).
\end{alignat*}
The result now follows once we apply Lemma \ref{lem_coupon} with $c = \frac{n}{\varepsilon^{n-1}}$.
\end{proof}

In the light of the following result, we see that the actual number of
samples required does not deviate significantly from the value provided
by Theorem \ref{thm_1}.
\begin{thm}
\label{thm_2}
Let $X$ be the number of samples observed before obtaining one in each region.
Then, for any constant $s > 0$ we have
\begin{equation*}
P(X > c \ln c + sc) < e^{-s}.
\end{equation*}
\end{thm}

\begin{proof}
The probability of not obtaining the $i$th region after $c \ln c + sc$ steps is
\begin{equation*}
\left(1 - \frac{1}{c}\right)^{c(\ln c + s)} < e^{-(\ln c + s)} = \frac{1}{e^s c}.
\end{equation*}
By a union bound, the probability that a region has not been obtained after
$c \ln c + s c$ steps is only $e^{-s}$.
\end{proof}
\par
Note that one can use a Chernoff bound to obtain an even tighter bound in Theorem \ref{thm_2} since
\[
\lim_{c\rightarrow\infty}P(X > c \ln c + sc) = 1 - e^{-e^{-s}}.
\]
See \cite{Mitzenmacher05} for details.
\par
We should note that in order to apply Lemma \ref{lem_coupon} the patches
used to cover $S_2^{n-1}$ should be disjoint which is clearly not the case
since we have used $B_2^{n-1}(\varepsilon)$ for this purpose.
This leads to an overestimate of the
samples needed to obtain a dense covering, and thus
the argument presented in this section is only a rough estimate.
However, as empirically demonstrated in the previous section,
a fixed number of samples as in \cite{Seol08} is definitely
not sufficient either.
To come up with a tight estimate of the number of sample points needed,
one has to express $S_2^{n-1}$ as a disjoint union of small patches.
The delicacy lies in the requirement that this has to be achieved
through a constructive process in a way that the surface
area of each patch can be explicitly computed as a function of the dimension $n$,
and a characteristic measure $\varepsilon$.
To the best of the authors' knowledge there is no systematic method
in the literature to achieve this.

\section{Conclusions}
\label{sec_conc}
In this paper, we investigated the theoretical and practical aspects of several Euclidean
norm approximations in the literature and showed that these are in fact special cases
of the weighted city-block norm. We evaluated the average and maximum errors of these
norms using numerical simulations. Finally, we demonstrated that the maximum errors
given in a recent study \cite{Seol08} are significantly optimistic.
\par
The implementations of the approximate norms described in this paper will be made publicly
available at \url{http://www.lsus.edu/faculty/~ecelebi/research.htm}.

\section{Acknowledgments}
\label{sec_ack}
This work was supported by a grant from the Louisiana Board of Regents (LEQSF2008-11-RD-A-12).
The authors are grateful to Changkyu Seol for clarifying various points about his paper.

\bibliographystyle{elsarticle-num}

\end{document}